\documentclass[leqno]{amsart}
\usepackage{times}
\usepackage{amsfonts,amssymb,amsmath,amsgen,amsthm}
\usepackage{hyperref}
\usepackage{color}

\theoremstyle{plain}
\newtheorem{theorem}{Theorem}[section]
\newtheorem{definition}[theorem]{Definition}
\newtheorem{assumption}[theorem]{Assumption}
\newtheorem{lemma}[theorem]{Lemma}
\newtheorem{corollary}[theorem]{Corollary}
\newtheorem{proposition}[theorem]{Proposition}

\theoremstyle{remark}
\newtheorem{remark}[theorem]{Remark}

\def\R{{\mathbf R}}
\def\N{{\mathbf N}}
\def\O{\mathcal O}
\def\F{\mathcal F}
\def\dd{\mathrm d}

\def\({\left(}
\def\){\right)}
\def\<{\left\langle}
\def\>{\right\rangle}
\def\le{\leqslant}
\def\ge{\geqslant}

\def\Eq#1#2{\mathop{\sim}\limits_{#1\rightarrow#2}}
\def\Tend#1#2{\mathop{\longrightarrow}\limits_{#1\rightarrow#2}}

\def\d{{\partial}}

\def\si{{\sigma}}

\DeclareMathOperator{\RE}{Re}
\DeclareMathOperator{\IM}{Im}

\numberwithin{equation}{section}

\begin{document}

\title[Large time in NLS with potential]{Large time behavior in
  nonlinear Schr\"odinger equation 
  with time dependent potential}
\author[R. Carles]{R\'emi Carles}
\address{CNRS \& Univ. Montpellier~2\\Math\'ematiques
\\CC~051\\34095 Montpellier\\ France}
\email{Remi.Carles@math.cnrs.fr}
\author[J. Drumond Silva]{Jorge Drumond Silva}
\address{Center for Mathematical Analysis, Geometry and Dynamical
  Systems\\ Departamento de Matem\'atica\\ Instituto Superior
  T\'ecnico\\ Av. Rovisco Pais\\ 1049-001 Lisboa\\Portugal}
\email{jsilva@math.ist.utl.pt}
\begin{abstract}
We consider the large time behavior of solutions to defocusing
nonlinear Schr\"odinger equation in the presence of a time dependent
external potential. The main assumption on the potential is that it
grows at most quadratically in space, uniformly with respect to the
time variable. We show a general exponential control of first order
derivatives and momenta, which yields a double exponential bound for
higher Sobolev norms and momenta. On the other hand, we show that if
the potential is an isotropic harmonic potential with a time dependent
frequency which decays sufficiently fast, then Sobolev norms
are bounded, and momenta grow at most polynomially in time, because the
potential becomes negligible for large time: there is scattering, even
though the potential is unbounded in space for fixed time. 
\end{abstract}
\thanks{This work was supported by the French ANR projects
  R.A.S. (ANR-08-JCJC-0124-01) and SchEq
  (ANR-12-JS01-0005-01). J. Drumond Silva was partially supported by
  the Center for Mathematical Analysis, Geometry and Dynamical
  Systems-LARSys through the Funda\c{c}\~ao para a Ci\^encia e
  Tecnologia (FCT/Portugal) Program POCTI/FEDER.} 
\maketitle

\section{Introduction}
\label{sec:intro}

\subsection{Motivation}
\label{sec:motiv}

For $x\in \R^d$, we consider the nonlinear Schr\"odinger equation with
a defocusing nonlinearity and a time dependent external potential:
\begin{equation}
  \label{eq:nlsp}
  i\d_t u +\frac{1}{2}\Delta u =V(t,x) u + |u|^{2\si}u;\quad u_{\mid t=0}=u_0.
\end{equation}
Throughout this paper, we make the following assumption on the
potential $V$:
\begin{assumption}\label{hyp:V}
  $V\in L^\infty_{\rm loc}(\R_t\times \R_x^d)$ is real-valued, and
  smooth with respect 
  to the space variable: for (almost) all $t\in \R$, $x\mapsto V(t,x)$
  is a $C^\infty$ map. Moreover, it is at most quadratic in space,
  \emph{uniformly with respect to time}: 
  \begin{equation*}
    \forall \alpha \in \N^d, \ |\alpha|\ge 2, \quad
    \d_x^\alpha V\in L^\infty(\R_t\times \R_x^d). 
  \end{equation*}
In addition, $t\mapsto \sup_{|x| \le 1 }|V(t,x)|$ belongs to $L^\infty(\R)$.
\end{assumption}
Observe that in this assumption --- a global in time version of the one originally imposed in \cite{Fujiwara79} --- the final condition is required to
ensure the boundedness in time of $V$ and its first order derivatives at points
within the unit ball to yield, after two integrations, the estimates
$|\nabla V(t,x)| \lesssim \< x \> $ and $|V(t,x)| \lesssim \<x\> ^2$,
uniformly for (almost) all $t$. This condition could, of course, be
equivalently substituted by demanding uniform boundedness in time for
$V$ and $\nabla V$ at fixed points in $\R_x^d$. It should also be
pointed out that no spectral properties of $V$ are imposed in this
assumption. 

A typical example that we have in mind is the time dependent harmonic
potential:
\begin{equation}
  \label{eq:harmo}
  V(t,x)= \frac{1}{2}\<Q(t)x,x\>,
\end{equation}
where the matrix $Q(t)\in \R^{d\times d}$ is real-valued, bounded and
symmetric. In this case, \eqref{eq:nlsp} appears for instance as an
envelope equation in the propagation of coherent states; see
\cite{CaFe11}. The model \eqref{eq:nlsp} with \eqref{eq:harmo} also
appears in Bose--Einstein condensation, typically 
for $\si=1$ (or $\si=2$ sometimes in the one-dimensional case $d=1$),
with $Q(t)$ a diagonal matrix; 
see e.g. \cite{CD96,GRPGV01,RKRP08}.

Throughout this paper, for $k\in \N$, we will denote by
\begin{equation*}
  \Sigma^k = \left\{ f\in L^2(\R^d)\ ;\ \| f \| _{\Sigma^k}:=
    \sum_{|\alpha|+|\beta|\le  k}\left\lVert x^\alpha \d_x^\beta
      f\right\rVert_{L^2(\R^d)}<\infty\right\}, 
\end{equation*}
and $\Sigma^1=\Sigma$. The main result in \cite{CaFe11} relies on the
property that the $\Sigma^k$ norm of $u$ solution to
\eqref{eq:nlsp}--\eqref{eq:harmo} grows at most exponentially in
time. This property has been established in some cases (see next
subsection), and we present here several extensions. 
\smallbreak

Since, except possibly for the potential $V$, the equation is invariant under
the transform $u(t,x)\mapsto \bar u(-t,x)$, from now on we consider 
\eqref{eq:nlsp} for $t\ge 0$ only.

\subsection{Known results}
\label{sec:known}

It has been proved in \cite{Ca11} that under Assumption~\ref{hyp:V}, with $\si>0$
and if the nonlinearity is energy-subcritical ($\si<2/(d-2)$ if $d\ge
3$), then for all $u_0\in \Sigma$, \eqref{eq:nlsp} has a unique,
local solution, such that $u,xu,\nabla u \in C((-T,T);L^2)\cap
L^{(4\si+4)/(d\si)}_{\rm  loc}(\R;L^{2\si+2})$. Moreover, its
$L^2$-norm is independent of time,
\begin{equation*}
  \|u(t)\|_{L^2}=\|u_0\|_{L^2},\quad \forall t\in (-T,T).
\end{equation*}
It is also shown in \cite{Ca11} that the only obstruction to global existence 
is the unboundedness of
$\|\nabla u(t)\|_{L^2}$ in finite time, a possibility which is ruled
out either if $\si<2/d$ or (when the nonlinearity is defocusing, which is the case in \eqref{eq:nlsp})
 if $V$ is $C^1$ in $t$ and $\d_t V$ satisfies
Assumption~\ref{hyp:V}; one can then let $T=\infty$ in the above
statements. We will prove in this paper that actually, $V$ need not be
$C^1$ in $t$: Assumption~\ref{hyp:V} seems to be the only relevant
hypothesis. 
\smallbreak 

Note that requiring symmetric properties in terms of regularity for
$xu$ and $\nabla u$ is 
natural, at least in the case of the linear harmonic potential, since
the harmonic oscillator rotates the phase space. More generally,
unless $\nabla V$ is 
bounded, \eqref{eq:nlsp} can be solved in $\Sigma$, but not merely in $H^s$,
no matter how large $s$ is; see \cite{CaCauchy}.
\smallbreak

In the energy critical case $\si=2/(d-2)$ with $V(x)= \epsilon |x|^2$
a time independent isotropic quadratic potential ($\epsilon=-1$ or $+1$), it
was proved in \cite{KiViZh09} that \eqref{eq:nlsp} has a unique global
solution in $\Sigma$, like in the case $V=0$ proved in
\cite{CKSTTAnnals,RV,VisanH1} (see also \cite{Vi12,KiVi-p}).  
\smallbreak

Concerning the large time behavior and norm growth of the solutions, few results are
available, and only for particular cases of harmonic potentials \eqref{eq:harmo}. 
If the nonlinearity is $L^2$-subcritical and smooth (an
assumption which boils down to the one-dimensional cubic case
$d=\si=1$), and $Q(t)$ (a real valued scalar function in $d=1$) is locally 
Lipschitz and remains bounded, 
then the Sobolev norms and the momenta of $u$ in $L^2$
grow at most exponentially in time under Assumption~\ref{hyp:V},
(\cite{Ca11}): if $u_0\in 
\Sigma^k$, there exists 
$C>0$ such that 
\begin{equation*}
  \|u(t)\|_{\Sigma^k}\le C {\rm e}^{Ct},\quad \forall t\ge 0. 
\end{equation*}
If the nonlinearity is $L^2$-critical or supercritical ($\si\ge 2/d$)
and the case of a time dependent isotropic repulsive
quadratic potential is  considered,
\begin{equation*}
  V(t,x)=\frac{1}{2}\Omega(t)|x|^2,\quad \text{with }\Omega(t)\le 0,
\end{equation*}
($\Omega(t)$ also locally Lipschitz) 
then the same exponential control is available (\cite{Ca11}). 
\smallbreak

We note that if
 $V(t,x)= -|x|^2$, the nonlinearity in \eqref{eq:nlsp} is
negligible for large time as there is scattering in $\Sigma$ (see \cite{CaSIMA} for
the energy-subcritical case, and \cite{KiViZh09} for the
energy-critical case), 
\begin{equation*}
\exists u_+\in \Sigma,\quad  \left\| u(t) - {\rm e}^{i\frac{t}{2}(\Delta
      +|x|^2)}u_+\right\|_{\Sigma}\Tend t {+\infty} 0,
\end{equation*}
and the solutions to
the linear equation (with potential) grow exponentially in time in the
space $\Sigma$, since
\begin{equation*}
  {\rm e}^{i\frac{t}{2}(\Delta
      +|x|^2)}u_+\Eq t{+\infty} \frac{1}{\sinh t}
    \F \(u_+{\rm e}^{|\cdot |^2/2}\)\(\frac{x}{\sinh t}\) {\rm e}^{i\frac{\cosh
     t}{\sinh t}\frac{|x|^2}{2}}, 
\end{equation*}
where we normalize the Fourier transform as
\begin{equation*}
  \F(f)(\xi)=
  \widehat
  f(\xi)=\frac{1}{(2i\pi)^{d/2}} \int_{\R^d} {\rm e}^{-ix\cdot \xi
  }f(x)\dd x. 
\end{equation*}
Scattering (in the $L^2$ topology) also holds for more general time dependent isotropic repulsive
quadratic potentials,
in the $L^2$-critical or supercritical ($\si\ge 2/d$)
cases (\cite{Ca11}).
Therefore, exponential growth of
Sobolev norms for solutions of \eqref{eq:nlsp} does occur in the presence of these 
repulsive quadratic potentials,
for the underlying reason that the corresponding linear solution has that property.

On
the other hand, if $V=0$ and $\si\ge 2/d$ is an integer, there is scattering
in $\Sigma$ to the free linear case, thus leading to bounded Sobolev norms
and momenta that grow polynomially (\cite{Wa04}): if $u_0\in
\Sigma^k$, there exists $C$ such that
\begin{equation*}
  \|u(t)\|_{H^k}\le C;\quad \|\lvert x\rvert^k u(t)\|_{L^2}\le
  C\<t\>^k,\quad \forall t\ge 0. 
\end{equation*}
\smallbreak

Finally, for confining harmonic potentials, one should start by noticing that 
in the time independent case
\begin{equation}
  \label{eq:confining}
  V(x) = \frac{1}{2} \sum_{j=1}^d \Omega_j x_j^2,\quad \text{with }\Omega_j > 0,
\end{equation} 
the conservation of energy (see
\eqref{eq:energy}--\eqref{eq:evolenergy} below) 
immediately implies
boundedness of the $\Sigma$ norm, $u \in L^{\infty}(\R;\Sigma)$. Moreover,
the existence of periodic solutions of the form $u(t,x)=e^{-i\omega t}\psi(x)$ 
to the nonlinear problem, with isotropic
confining harmonic potential $\Omega_j =\Omega >0$ (see \cite{CaRotating} for details),
as well as the linear 
dynamics (which is time-periodic), 
naturally lead to the conjecture that we may also have 
$u \in L^{\infty}(\R;\Sigma^k)$, at least for localized and smooth enough data.
In the case of the one-dimensional harmonic time independent
oscillator $V(t,x)=x^2$, standard techniques yield an exponential
control. Such bounds have been improved in
\cite{GrTh11} for small perturbations of $x^2$, by adapting methods
from finite dimensional 
dynamical systems, to prove that at least for small initial data, the
$\Sigma^k$-norm may be bounded, because the solution is quasi-periodic
in time
(\cite{GrTh11}). Such a conclusion is therefore expected to
remain valid in a rather general setting for confining potentials. See
also \cite{BuThTz-p} for results in this direction. 
\smallbreak

A different perspective consists in considering the case where the
potential decays rapidly in time. Such a case has been considered for
potentials which are exactly quadratic in space:
\begin{proposition}[From Proposition 1.9 and Lemma 4.3 in \cite{Ha13}]\label{prop:lysianne}
  Let $1\le d\le 3$, $\si\in \N$ with $\si=1$ if $d=3$. Suppose that
  $V$ is of the form \eqref{eq:harmo}, with
  \begin{equation}
    \label{eq:lysianne}
    \left\lvert Q(t)\right\rvert +
    \<t\> \left\lvert \frac{\dd }{\dd t}Q(t)\right\rvert \le
    \frac{C}{\<t\>^\gamma},\quad \text{for some }\gamma>2.
  \end{equation}
If $u_0\in \Sigma^k$ ($k\in \N$, $k\ge 1$), then there exist
$\eta,C>0$ such that 
\begin{equation*}
 \|u(t)\|_{H^1}\le C,\quad \|xu(t)\|_{L^2}\le
 C\<t\>^{1+\eta},\quad\|u(t)\|_{\Sigma^k}\le C 
 {\rm e}^{Ct}, \quad \forall t\ge 0.  
\end{equation*}
\end{proposition}
As a matter of fact, only the cubic nonlinearity case ($\si=1$), in dimensions $d=2$ or $3$, is
considered in \cite{Ha13}, but the proof remains valid under the above
assumptions.

\subsection{New results}
\label{sec:new}

\begin{theorem}\label{theo:exp}
  Let $d\ge 1$, $\si>0$, with $\si<2/(d-2)$ if $d\ge 3$. If $V$ satisfies
  Assumption~\ref{hyp:V} and $u_0\in
  \Sigma$, then the solution $u$ to \eqref{eq:nlsp} is global in time:
  \begin{equation*}
    u,\nabla u, xu \in C(\R;L^2(\R^d)). 
  \end{equation*}
Moreover, it grows at most exponentially
  in time: there exists $C>0$ such that
\begin{equation*}
 \|u(t)\|_{\Sigma}\le C {\rm e}^{Ct}, \quad \forall t\ge 0.  
\end{equation*}
\end{theorem}
Note that in general, this bound is (qualitatively) sharp, as shown by
the repulsive harmonic potential case
 $V(t,x)= -|x|^2$ mentioned above. Unlike the previously known results,
 this norm growth conclusion is not restricted to harmonic potentials only.
Using Strichartz estimates, we infer the following corollary, concerning the growth rate of
the higher order $\Sigma^k$ norms:
\begin{corollary}[Double exponential bound]\label{cor:doubleexp}
  Let $d\ge 1$, $k\ge 2$, $\si>0$ with $\si<2/(d-2)$ if $d\ge 3$. Suppose that
  the map $z\mapsto |z|^{2\si}z$ is $C^k$. If $u_0\in \Sigma^k$, then
  there exists $C>0$ such that 
  \begin{equation*}
    \sup_{2\le |\alpha|+|\beta|\le k}\|x^\alpha \d_x^\beta
    u(t)\|_{L^2}\le C {\rm e}^{{\rm e}^{Ct}}, \quad \forall t\ge 0.  
  \end{equation*}
\end{corollary}
\begin{remark}
	For time independent confining harmonic potentials \eqref{eq:confining} we have boundedness 
	of the
	$\Sigma$ norm of the global solutions $u \in L^{\infty}(\R;\Sigma)$, rather than the general
	exponential growth of Theorem~\ref{theo:exp}. Given this better starting point 
	for the lower order
	derivatives and momenta, and using exactly the same method of proof by induction
	as in this corollary, we obtain then an exponential
	bound for the higher order norms rather than the double exponential  
	  \begin{equation*}
	  	\sup_{2\le |\alpha|+|\beta|\le k}\|x^\alpha \d_x^\beta
	  	u(t)\|_{L^2}\le C {{\rm e}^{Ct}}, \quad \forall t\ge 0.  
	  \end{equation*}
        See Remark~\ref{remark:conf} for details.
\end{remark}
\begin{remark}
  In the case of the nonlinear Schr\"odinger equation without
  potential ($V=0$), one can infer similarly that the $\dot H^k$ ($k\ge 2$)
  norm of solutions which are globally bounded in $H^1(\R^d)$ grows at
  most exponentially in time. The use of Bourgain spaces (as initiated
  in \cite{Bo96,St97}) makes it possible to soften this exponential
  bound to a polynomial bound. However, adapting these spaces to the
  present framework (which, in addition, is not Hamiltonian if $\d_t V\not =0$)
  seems to be a rather challenging issue. 
\end{remark}
\begin{remark}
  Another strategy might consist in resuming the pseudo-energy used in
  \cite{RaSz09}. Note however that the
  pseudo-energy introduced in \cite{RaSz09} turns out to be helpful in
  the context of the analysis of blowing-up solutions. Even in the
  absence of an external potential ($V=0$), we have not been able to
  adjust this pseudo-energy to prove the boundedness of the $H^2$-norm
  of $u$ (nor even an exponential control), a property which is known
  by other arguments.  
\end{remark}
Using a (global) lens transform, we prove the following result, to be
compared with Proposition~\ref{prop:lysianne}. 
\begin{theorem}\label{theo:decay}
    Let $d\le 3$, and $\si\in \N$, with $\si\ge 2/d$ and $\si=1$ if
    $d=3$. Suppose that 
  $V$ is of the form 
  \begin{equation}
    \label{eq:lysianne2}
    V(t,x)=\frac{1}{2}\Omega(t)|x|^2,\quad\text{with}\quad
    \left\lvert \Omega(t)\right\rvert \le 
    \frac{C}{\<t\>^\gamma}\quad \text{for some }\gamma>2.
  \end{equation}
If $u_0\in \Sigma^k$, then there exists $C>0$ such that
\begin{equation}\label{eq:pol1}
 \|u(t)\|_{H^k}\le C,\quad \|\<x\>^ku(t)\|_{L^2}\le C \<t\>^k
 , \quad \forall t\ge 0.  
\end{equation}
Finally, if $u_0\in \Sigma$, then there exists $u_+\in\Sigma$ such
that
\begin{equation*}
  \left\|u(t)- {\rm e}^{i\frac{t}{2}\Delta}u_+\right\|_{L^2}\Tend t {+\infty} 0. 
\end{equation*}
\end{theorem}
\begin{remark}
  A consequence of the proof of this result is that if the potential
satisfies \eqref{eq:lysianne2}, the Strichartz estimates associated to
the linear evolution are \emph{global in time} (while, as recalled
above, this is not the case if $V(t,x)=|x|^2$).
\end{remark}

Compared to Proposition~\ref{prop:lysianne}, our assumptions seem to
be  more stringent on two aspects:
\begin{itemize}
\item The matrix $Q$ is of the form $Q(t)=\Omega(t){\rm I}_d$, i.e. we consider 
isotropic potentials only.
\item The nonlinearity is $L^2$-critical or $L^2$-supercritical
  ($\si\ge 2/d$). 
\end{itemize}
However, it turns out that the second point rules out only one case compared
to Proposition~\ref{prop:lysianne}, and that is when $d=\si=1$, for which
exponential bounds in $\Sigma^k$ for all $k$ were already known under the mere
assumption that $\Omega$ is bounded (\cite{Ca11}). 

On the other hand, our assumptions demand only a certain minimum decay in time for $\Omega$ and
impose no restriction on its time derivative (which, in our case, might not even exist). So a rapidly oscillatory potential for large
time as, for instance, is the case with 
\begin{equation*}
  \Omega(t)=\frac{\cos\({\rm e}^t\)}{\<t\>^3}
\end{equation*}
is eligible for Theorem~\ref{theo:decay}, while it is not for
Proposition~\ref{prop:lysianne}. In fact, it does seem more natural to require a
decay exclusively on the function $\Omega$, rather than also adding
conditions for its time derivative, as we will see from the
proof of  Theorem~\ref{theo:decay}: in the linear case 
\begin{equation*}
  i\d_t u+\frac{1}{2}\Delta u = \frac{1}{2}
  \Omega(t)|x|^2 u,
\end{equation*}
if $|\Omega(t)|\lesssim \<t\>^{-\gamma}$ for $\gamma>2$,
then $u$ satisfies \eqref{eq:pol1}. As suggested by the last statement
of the proposition, the potential is negligible for large time, even
though for fixed $t$, the harmonic potential cannot be treated as a
perturbation. Heuristically, this can be seen through the standard
asymptotics (in $L^2$)
\begin{equation}\label{eq:scatt}
  {\rm e}^{i\frac{t}{2}\Delta}f \Eq t {+\infty} \frac{1}{t^{d/2}}\widehat
  f\(\frac{x}{t}\) {\rm e}^{i|x|^2/(2t)}.
\end{equation}
Asymptotically, the right variable is $x/t$, and since by assumption
\begin{equation*}
  \left|\Omega(t)\right|\lvert x\rvert^2\lesssim
  \left|\frac{x}{t}\right|^2\frac{1}{t^{\gamma-2}},
\end{equation*}
it is sensible to expect the external potential to be negligible for
large time. The proof of Theorem~\ref{theo:decay} will make this
intuition more precise. Also, note that compared to the conclusion of
Proposition~\ref{prop:lysianne}, the control of the momenta (even in
the case of $\|xu\|_{L^2}$) and higher
Sobolev norms is improved. The sharpness of the decay assumption on
$\Omega$ is discussed in Remark~\ref{rem:sharp}.

\subsection{Outline of the paper}
\label{sec:outline}

Theorem~\ref{theo:exp} is proved in Section~\ref{sec:general}, and
we infer Corollary~\ref{cor:doubleexp} in
Section~\ref{sec:double}. The case where $V$ is an isotropic harmonic
potential \eqref{eq:lysianne2} is treated in
Section~\ref{sec:decay}, where Theorem~\ref{theo:decay} is
established.

\section{Proof of Theorem~\ref{theo:exp}}
\label{sec:general}

First, we recall that from \cite{Ca11}, under the assumptions of
Theorem~\ref{theo:exp}, \eqref{eq:nlsp} has a unique, local
solution. The obstruction to global existence is the unboundedness of
$\|\nabla u(t)\|_{L^2}$ in finite time. Thus, we simply have to prove
a suitable \emph{a priori} estimate. 
\smallbreak

A natural candidate for an energy in the case of \eqref{eq:nlsp}
is 
\begin{equation}\label{eq:energy}
  E(t)=\frac{1}{2}\|\nabla u(t)\|_{L^2}^2
  +\frac{1}{\si+1}\|u(t)\|_{L^{2\si+2}}^{2\si+2} +
  \int_{\R^d} V(t,x) \lvert u(t,x)\rvert^2\dd x. 
\end{equation}
It was established in \cite{Ca11} that if  $V$ is $C^1$ with respect
to $t$, and $\d_t V$ satisfies 
  Assumption~\ref{hyp:V}, then $E\in
  C^1((-T,T);\R)$, and its evolution is given by
  \begin{equation}\label{eq:evolenergy}
    \frac{\dd E}{\dd t} = \int_{\R^d} \d_t V(t,x) \lvert
    u(t,x)\rvert^2\dd x.
  \end{equation}
In the same spirit as in \cite{AnSp10}, introduce the pseudo-energy
\begin{equation*}
  {\mathcal E}(t) = \frac{1}{2}\|\nabla u(t)\|_{L^2}^2
  +\frac{1}{\si+1}\|u(t)\|_{L^{2\si+2}}^{2\si+2} +
  \frac{1}{2}\int_{\R^d} |x|^2 \lvert u(t,x)\rvert^2\dd x. 
\end{equation*}
From the relation
\begin{equation*}
  {\mathcal E}(t) = E(t) +\frac{1}{2}\int_{\R^d} \(|x|^2-2V(t,x)\) \lvert
  u(t,x)\rvert^2\dd x,
\end{equation*}
we infer, at least formally,
\begin{align*}
 \frac{\dd {\mathcal E}}{\dd t} &= \frac{1}{2}\int_{\R^d}
 \(|x|^2-2V(t,x)\) \d_t \lvert
  u(t,x)\rvert^2\dd x\\
&= \RE \int_{\R^d}
 \(|x|^2-2V(t,x)\) \bar u(t,x) \d_t 
  u(t,x)\dd x\\
&= \IM \int_{\R^d}
 \(|x|^2-2V(t,x)\) \bar u(t,x) i\d_t 
  u(t,x)\dd x\\
&= -\frac{1}{2}\IM \int_{\R^d}
 \(|x|^2-2V(t,x)\) \bar u(t,x) 
  \Delta u(t,x)\dd x\\
&= \IM \int_{\R^d}
\bar u(t,x)  \(x -\nabla V(t,x)\)\cdot  
  \nabla u(t,x)\dd x.
\end{align*}
Now from Assumption~\ref{hyp:V} and the observations that follow it, there exists $C$ independent of $t$
such that
\begin{equation*}
  |\nabla V(t,x)|\le C \<x\>. 
\end{equation*}
Therefore, using the conservation of mass and Cauchy--Schwarz
inequality, we infer
\begin{equation*}
   \frac{\dd {\mathcal E}}{\dd t} \le C_0\( 1+ \|x
   u(t)\|_{L^2}\|\nabla u(t)\|_{L^2}\) .
\end{equation*}
From Young's inequality, ${\mathcal E}$ satisfies an inequality of the
form $\dot {\mathcal E} \le C_0(1+{\mathcal E})$, with $C_0$
independent of $t$ (but depending on the conserved mass
$\| u_0 \|_{L^2}$): Gronwall lemma yields an exponential bound. 
\smallbreak

Under the assumptions of Theorem~\ref{theo:exp} though, $V$ need not
be differentiable with respect to time, so the above computations cannot
be followed step by step. To overcome this issue, simply note that the
evolution of $E$ was merely used as a shortcut in the above
presentation, and that by using standard arguments (see
e.g. \cite{CazCourant}), one directly proves 
\begin{equation*}
  \frac{\dd {\mathcal E}}{\dd t} =\IM \int_{\R^d}
\bar u(t,x)  \(x -\nabla V(t,x)\)\cdot  
  \nabla u(t,x)\dd x,\quad \forall t\in (-T,T),
\end{equation*}
and Theorem~\ref{theo:exp} follows.

\section{Double exponential control: proof of
  Corollary~\ref{cor:doubleexp}} 
\label{sec:double}

Corollary~\ref{cor:doubleexp} is proved by induction on $k$, applying the
following lemma to the inequalities that result from the use of Strichartz 
estimates for the evolution equations of
$x^\alpha \d_x^\beta u$. In this lemma, thus, $\|w\|$ must be thought of as a placeholder for
all the combinations of norms of $x^\alpha \d_x^\beta u$ of order $k=|\alpha|+|\beta|$. 
\begin{lemma}
  \label{lemma:main}
  For $0\le s \le t$, denote by $\tau=t-s$ and $I=[s,s+\tau]=[s,t]$. Let $w$ satisfy the following property: there exist Lebesgue exponents $p,q\ge 1$, parameters 
  $\alpha, \tau_0>0$, a non-decreasing function $f$ and a constant $C$ such that, 
  given any $s\ge 0, \ \tau \in [0,\tau_0],$ 
  \begin{equation}
    \label{eq:induction}
    \|w\|_{L^p(I;L^q)\cap L^{\infty}(I;L^2)}\le C \|w(s)\|_{L^2} + C\tau^\alpha {\rm e}^{Ct}
    \|w\|_{L^p(I;L^q)\cap L^{\infty}(I;L^2)}+f(t). 
  \end{equation}
Then there exists $C_1$, depending only on $C$, $\alpha$ and $\tau_0$, but independent 
of $t\ge 0$, such that
\begin{equation*}
  \|w\|_{L^p([0,t];L^q)\cap L^{\infty}([0,t];L^2)}\le C_1{\rm e}^{{\rm e}^{C_1t}}
  (\|w(0)\|_{L^2}
  +f(t)),\quad \forall
  t\ge 0. 
\end{equation*}
\end{lemma}
\begin{proof}
	Let us consider the interval $[0,t]$ and fix this $t$ as an upper bound for the time variable.
	
	Then, for this value of $t$, we take $\tau$ satisfying
	\begin{equation}
	\label{eq:tau}
	C\tau^\alpha {\rm e}^{Ct}= \frac{1}{10} \Leftrightarrow \tau = \left(\frac{1}{C 10}\right)^{\frac{1}{\alpha}}  {\rm e}^{-\frac{Ct}{\alpha}},
	\end{equation}
	(without loss of generality, the total time $t$ can be chosen initially to be large enough to have $\tau \le \tau_0$ as well as $\tau \le t$)
	such that the corresponding term in \eqref{eq:induction} gets absorbed by the left hand side of the inequality, as
	\begin{equation}
	\label{eq:bound}
	\|w\|_{L^p(I;L^q)\cap L^{\infty}(I;L^2)}\le \frac{10}{9} \Big(C \|w(t')\|_{L^2}+f(t'+\tau)\Big),
	\end{equation}
	for any interval $I=[t',t'+\tau] \subset [0,t]$.
	
	Now, breaking up the full interval $[0,t]$ into $N \sim t/\tau$ small intervals of length $\tau$,
	$I_j=[t_j,t_{j+1}]=[t_j,t_j + \tau]$, $j=0, \ldots, N-1$, such that 
        $[0,t]=\cup_j I_j$, we have, for $1\le p < \infty$,    
        \begin{multline*}
	  \|w\|_{L^p([0,t];L^q)}^p = \int_0^t \|w\|_{L^q}^p \dd t'=
          \sum_{j=0}^{N-1} \int_{t_j}^{t_j+\tau} \|w\|_{L^q}^p \dd t'= 
          \sum_{j=0}^{N-1} \|w\|_{L^p(I_j;L^q)}^p \\
          \le \sum_{j=0}^{N-1} \|w\|_{L^p(I_j;L^q)\cap L^{\infty}(I_j;L^2)}^p \le
          \( \frac{20}{9}  \)^p \; \sum_{j=0}^{N-1} \big(C^p
          \|w(t_j)\|_{L^2}^p + f(t_{j+1})^p 
          \big),
        \end{multline*}
	We now use the fact that the
	$\|w\|_{L^{\infty}(I_{j-1};L^2)}$ norm, from the previous time
        step $j-1$, bounds the  
        $\|w(t_j)\|_{L^2}$ norm at the initial time of the following one, 
     to obtain from \eqref{eq:bound}
     \begin{equation*}
     	\|w(t_j)\|_{L^2}^p \le \|w\|_{L^p(I_{j-1};L^q)\cap L^{\infty}(I_{j-1};L^2)}^p 
     	\le \left(\frac{20}{9}\right)^p \Big(C^p \|w(t_{j-1})\|_{L^2}^p+f(t_j)^p\Big).
     \end{equation*}    

     Applying this estimate repeatedly and bounding all the terms $f(t_j)$ uniformly
     by their maximum $f(t)$ over the whole interval $[0,t]$, we infer
     \begin{align*}
     	\|w(t_j)\|_{L^2}^p &\le \left(\frac{20}{9} C\right)^{jp} \|w(0)\|_{L^2}^p +
     	\left(1+\left(\frac{20}{9} C\right)^{p}+\dots+\left(\frac{20}{9} C\right)^{(j-1)p}
     	\right) \left(\frac{20}{9}\right)^p f(t)^p\\
     	&\le \tilde{C} \left(\frac{20}{9} C\right)^{jp} \big(\|w(0)\|_{L^2}^p
     	+f(t)^p),
     \end{align*}
     so that
	\begin{equation*}
	 \|w\|_{L^p([0,t];L^q)}^p \le \( \frac{20}{9} \)^p \; 
         \sum_{j=0}^{N-1} C^p \tilde{C} \left(\frac{20}{9}
           C\right)^{jp} \big(\|w(0)\|_{L^2}^p 
         +f(t)^p) + \( \frac{20}{9} \)^p N f(t)^p,
	\end{equation*}
	thus yielding
	\begin{equation*}
         \|w\|_{L^p([0,t];L^q)} \le \tilde{C} \( \frac{20}{9} C \)^{N}
         (\|w(0)\|_{L^2} +f(t)), 
	\end{equation*}
	for a new $\tilde{C}$ constant. Finally, using the fact that
	$N \sim t/\tau$ and \eqref{eq:tau}, one obtains
	\begin{equation*}
	 \|w\|_{L^p([0,t];L^q)} \le C_1 {\rm e}^{{\rm e}^{C_1t}}(\|w(0)\|_{L^2}
	 +f(t)),
	\end{equation*}
        with $C_1$, a final constant, as in the statement of the lemma.

	Of course the case $p=\infty$ is even simpler, as
	\begin{equation*}
	 \|w\|_{L^{\infty}([0,t];L^q)}=\max_{j=0,\ldots,N-1} \|w\|_{L^{\infty}(I_j;L^q)}
	  \le \frac{10}{9}  \Big( C \max_{j=0,\ldots,N-1}  \|w(t_j)\|_{L^2} + f(t)\Big),
	\end{equation*}
	and, as above, each $\|w(t_j)\|_{L^2}$ norm can be controlled
        by the previous time 
	step $\|w\|_{L^{\infty}(I_{j-1};L^2)}$ norm. So that, repeated application again of
	\eqref{eq:bound} yields
	\begin{equation*}
	    \|w\|_{L^{\infty}([0,t];L^q)} \le  \tilde{C} \( \frac{10}{9} C \)^{N} \Big(\|w(0)\|_{L^2}+f(t)\Big),
	\end{equation*}
	from which the double exponential bound for $p=\infty$ now
        follows as before. 

	To conclude, we only need to note that the
        $\|w\|_{L^{\infty}([0,t];L^2)}$ norm falls into 
	this latter case.
\end{proof}

\begin{remark}
	Observe that, if some rate of growth in time is already known
        for the $\|w(t)\|_{L^2}$ norm --- as in the case of the
        conservation of the mass or the exponential growth of the
        $\Sigma$  
	norm in Theorem~\ref{theo:exp}, for example --- then the
        previous proof can be greatly  
	simplified, bounding all of the $\|w(t_i)\|_{L^2}$ uniformly
        by $\|w(t)\|_{L^2}$, to yield  
	the following
	exponential control of the $\|w\|_{L^p([0,t];L^q)}$ norm from
        its corresponding  
	$\|w(t)\|_{L^2}$ norm growth, as well as $f(t)$,
	\begin{equation*}
		\|w\|_{L^p([0,t];L^q)} \le C_1 {\rm e}^{C_1t} \big(\|w(t)\|_{L^2}+f(t)\big). 
	\end{equation*}
\end{remark}

We now return to the main proof of Corollary~\ref{cor:doubleexp} by observing that, although 
Lemma~\ref{lemma:main} imposes no restrictions on the pair $(p,q)$ of Lebesgue exponents, the Strichartz estimates require only admissible pairs suited to each particular equation which,
for the Schr\"odinger equation (with or without potential), are the following.

\begin{definition}\label{def:adm}
 A pair $(p,q)$ is admissible if $2\le q
  <\frac{2d}{d-2}$ ($2\le q\le\infty$ if $d=1$, $2\le q<
  \infty$ if $d=2$)  
  and 
$$\frac{2}{p}=\delta(q):= d\left( \frac{1}{2}-\frac{1}{q}\right).$$
\end{definition}

Strichartz estimates for the Schr\"odinger equation with potential satisfying Assumption~\ref{hyp:V} are a consequence of the results in \cite{Fujiwara79}, \cite{Fujiwara}.
Indeed, the existence of a strongly continuous propagator, unitary on $L^2$, for the linear
Schr\"odinger equation with potential satisfying Assumption~\ref{hyp:V} is proved in \cite{Fujiwara79}. In \cite{Fujiwara} it is proved that, for bounded time intervals, this
propagator exhibits an $L^1 - L^{\infty}$ decay in time. As is now well known, these two properties are the crucial ingredients that lead to Strichartz estimates for the linear
propagator (see e.g. \cite{KeelTao}). A precise statement of these estimates, in our
context, can be found in \cite[Section~2]{Ca11}.

Two points need careful attention at this point, though. The first one
is that these Strichartz 
estimates are just local in time. Unlike the case of the Schr\"odinger
equation without potential, 
whose estimates are global, only Strichartz estimates
for finite time intervals can be expected when general potentials are
considered. The typical counter-example is the linear 
Schr\"odinger equation with a confining harmonic potential
\eqref{eq:confining}, which 
exhibits time periodic solutions and thus cannot possibly satisfy
global dispersive estimates in time. The  
second point requiring a careful observation has to do with the fact
that the potentials being  
considered here also depend on time. The equation is therefore not
autonomous and the propagator  
depends now on the initial time of the flow. In particular, the
maximum size of the 
finite time interval $[s,s+\tau]$ for the 
Strichartz estimates to hold (i.e. the parameter $\tau_0$ of the
previous lemma) should thus depend generally on the overall time
interval $[0,t]$ being considered. 
It can be shown, however, that for potentials whose spatial derivatives are 
uniformly bounded in time --- which is the case
we are imposing in Assumption~\ref{hyp:V}, unlike the more frequent
condition of just  
local boundedness in time found in the literature, as in
\cite{Fujiwara79}, \cite{Fujiwara} or 
\cite{Ca11} --- one can indeed pick a uniform value of $\tau_0$ that
holds globally on $[0,\infty)$. 
See again \cite[Section~2]{Ca11} for a careful 
discussion of these two issues.

We now have all the properties needed to go ahead with the induction procedure. For the sake of 
clarity, we will start with the cases $k=0$ and $k=1$, for which an exponential growth is already 
known from Theorem~\ref{theo:exp}. However, the point is that we can easily write all 
the complete formulas and 
estimates for these simpler cases, which illustrate the essential features 
that remain for general higher values of $k$, whose computations and formulas then become much 
more cumbersome .

The cases $k=0$ and $k=1$ correspond to the same estimates used to prove local well posedness in 
$\Sigma$. The solution $u$, whose global existence has been established in Theorem~\ref{theo:exp},
is thus a fixed point of the Duhamel formulation for any initial time $t_0=s$. It satisfies, therefore, 
\begin{equation}
\label{eq:duhamel}
	u(t)=U(t,s)u(s) - i \int_{s}^{t} U(t,t')(|u|^{2\sigma} u) (t')\dd t',
\end{equation}
where $U(t,s)$ represents the linear propagator of \eqref{eq:nlsp} from time $s$ to time $t$. We
then pick the usual Lebesgue exponents
\begin{equation*}
	q=2\sigma+2; \qquad p=\frac{4\sigma+4}{d\sigma}; \qquad 
	\theta=\frac{2\sigma(2\sigma+2)}{2-(d-2)\sigma}.
\end{equation*}
This choice of the pair $(p,q)$ is admissible and we also have
\begin{equation*}
	\frac{1}{q'}=\frac{2\sigma}{q}+\frac{1}{q}; \qquad
	\frac{1}{p'}=\frac{2\sigma}{\theta}+\frac{1}{p},
\end{equation*}
so that applying Strichartz estimates on \eqref{eq:duhamel}, as long as the finite time 
interval is smaller than the uniform bound $|t-s|\le \tau_0$, we obtain
\begin{equation*}
	\|u\|_{L^p([s,t];L^q) \cap {L^{\infty}([s,t];L^2)}}\le 
	C\|u(s)\|_{L^2}+C \| |u|^{2\sigma}u\|_{L^{p'}([s,t];L^{q'})},
\end{equation*}
and then using H\"older to handle the nonlinear term,
\begin{eqnarray*}
	\|u\|_{L^p([s,t];L^q) \cap {L^{\infty}([s,t];L^2)}} & \le & 
	C\|u(s)\|_{L^2}+C\| u\|_{L^{\theta}([s,t];L^q)}^{2\sigma} \|u\|_{L^{p}([s,t];L^{q})} \\
	&\le& C\|u(s)\|_{L^2}+C |t-s|^{\frac{2\sigma}{\theta}}
	\| u\|_{L^{\infty}([s,t];L^q)}^{2\sigma} \|u\|_{L^{p}([s,t];L^{q})}.
\end{eqnarray*}
Finally, the $H^1$ subcritical condition $\sigma < 2/(d-2)$ permits the use of the Sobolev embedding,
from which we get
\begin{equation*}
	\|u\|_{L^p([s,t];L^q) \cap {L^{\infty}([s,t];L^2)}}  \le
	C\|u(s)\|_{L^2}+C |t-s|^{\frac{2\sigma}{\theta}}
	\| u\|_{L^{\infty}([s,t];H^1)}^{2\sigma} \|u\|_{L^{p}([s,t];L^{q})}.
\end{equation*}

Now, for $k=1$, one needs to develop similar estimates for $\nabla u$ and $x u$. Therefore,
we start by differentiating $\eqref{eq:nlsp}$ to obtain the evolution equation for
$\nabla u$ and its corresponding Duhamel formula
\begin{equation*}
	\nabla u(t) = U(t,s)\nabla u(s) - i \int_{s}^{t} U(t,t')\nabla(|u|^{2\sigma} u) (t')\dd t'
	-i \int_{s}^{t} U(t,t')\big(\nabla V(t') u(t')\big) \dd t'.
\end{equation*}
The only novelty now is the second integral, with the term $\nabla V(t')$, because all the 
remaining terms are estimated 
exactly as in the previous $k=0$ case. Assumption~\ref{hyp:V} implies that $|\nabla V|\lesssim 
\< x \>$ uniformly for all time, and noting that $(1,2)$ are conjugate exponents 
to the admissible pair $(\infty,2)$, we then get
\begin{eqnarray*}
	\|\nabla u\|_{L^p([s,t];L^q) \cap {L^{\infty}([s,t];L^2)}} & \le & 
	C\|\nabla u(s)\|_{L^2}+C\| u\|_{L^{\theta}([s,t];L^q)}^{2\sigma} 
	\|\nabla u\|_{L^{p}([s,t];L^{q})}  \\
	& & + \;C \|u \nabla V\|_{L^{1}([s,t];L^{2})} \\
	&\le& C\|\nabla u(s)\|_{L^2}+C |t-s|^{\frac{2\sigma}{\theta}}
	\| u\|_{L^{\infty}([s,t];H^1)}^{2\sigma} \|\nabla u\|_{L^{p}([s,t];L^{q})} \\
	& & + \;C |t-s| \|u \|_{L^{\infty}([s,t];L^{2})}
	    + C |t-s| \|x u \|_{L^{\infty}([s,t];L^{2})}.
\end{eqnarray*}

Analogously, for the momentum $xu$,
\begin{equation*}
	x u(t) = U(t,s) (x u)(s) - i \int_{s}^{t} U(t,t') (|u|^{2\sigma} x u) (t')\dd t'
	-i \int_{s}^{t} U(t,t') (\nabla u(t')) \dd t',
\end{equation*}
where the first order derivative $\nabla u$ in the second integral now appeared from writing 
\begin{equation*}
	x\frac12 \Delta u= \frac12 \Delta(xu) - \nabla u,
\end{equation*}
when multiplying the whole equation \eqref{eq:nlsp} by $x$, to obtain the evolution equation for the momentum. And 
following the same procedure as above
\begin{eqnarray*}
	\|x u\|_{L^p([s,t];L^q) \cap {L^{\infty}([s,t];L^2)}} & \le & 
	 C\|xu(s)\|_{L^2}+C |t-s|^{\frac{2\sigma}{\theta}}
	\| u\|_{L^{\infty}([s,t];H^1)}^{2\sigma} \|x u\|_{L^{p}([s,t];L^{q})} \\
	& & + \;C |t-s| \|\nabla u \|_{L^{\infty}([s,t];L^{2})}.
\end{eqnarray*}
So that, summing up the estimates for the first derivative and momentum,
\begin{eqnarray*}
	\|(xu,\nabla u)\|_{L^p([s,t];L^q) \cap {L^{\infty}([s,t];L^2)}} & \le & 
	C\|(xu,\nabla u)(s)\|_{L^2}\\
	& & + \;C |t-s|^{\frac{2\sigma}{\theta}}
	\| u\|_{L^{\infty}([s,t];H^1)}^{2\sigma} \|(xu,\nabla u)\|_{L^{p}([s,t];L^{q})} \\
	& & + \;C |t-s| \|(xu, \nabla u) \|_{L^{\infty}([s,t];L^{2})}\\
	& & + \;C |t-s| \| u \|_{L^{\infty}([s,t];L^{2})}.
\end{eqnarray*}

For $k=2$, the first order of momenta and spatial derivatives for which we are really getting
new information about its norm growth, let us denote by $\|w_2\|$ the sum of all corresponding
norms of terms of order 2
\begin{equation*}
	\|w_2\|=\sum_{|\alpha|+|\beta|=2} \|x^{\alpha} \partial^{\beta} u \|.
\end{equation*}
Then, after estimating the corresponding Duhamel formulations as we have done above, we get
\begin{eqnarray*}
	\|w_2\|_{L^p([s,t];L^q) \cap {L^{\infty}([s,t];L^2)}} & \le & 
	C\|w_2(s)\|_{L^2}\\
	& & + \;C |t-s|^{\frac{2\sigma}{\theta}}
	\| u\|_{L^{\infty}([s,t];H^1)}^{2\sigma} \|w_2\|_{L^{p}([s,t];L^{q})} \\
	& & + \;C |t-s| \|w_2 \|_{L^{\infty}([s,t];L^{2})}\\
	& & + \;C |t-s|^{\frac{2\sigma-1}{\theta}}
	\| u\|_{L^{\infty}([s,t];H^1)}^{2\sigma-1} \|\nabla u\|_{L^{p}([s,t];L^{q})}^2 \\
	& & + \;C |t-s| \|(xu, \nabla u) \|_{L^{\infty}([s,t];L^{2})}\\
	& & + \;C |t-s| \| u \|_{L^{\infty}([s,t];L^{2})},
\end{eqnarray*}
where the new term $|t-s|^{\frac{2\sigma-1}{\theta}}
	\| u\|_{L^{\infty}([s,t];H^1)}^{2\sigma-1} \|\nabla u\|_{L^{p}([s,t];L^{q})}^2$ occurs
from differentiating the nonlinear powers $|u|^{2\sigma}u$ twice, in \eqref{eq:nlsp}, for
the evolution equations of the second order derivatives of $u$.

Generally, then, for $k\ge 2$, if we write
\begin{equation*}
	\|w_k\|=\sum_{|\alpha|+|\beta|=k} \|x^{\alpha} \partial^{\beta} u \|,
\end{equation*}
after writing evolution equations for each of these $k$ order momenta and spatial 
derivatoves, obtained by differentiating and multiplying \eqref{eq:nlsp} by enough powers of $x$,
and finally estimating the corresponding Duhamel formulations, we finally obtain 
\begin{eqnarray}
	\lefteqn{\|w_k\|_{L^p([s,t];L^q) \cap {L^{\infty}([s,t];L^2)}} \le  
	C\|w_k(s)\|_{L^2}}  \nonumber \\ 
        \label{eq:generalk2} 
	& & \hspace{1cm} + \;C |t-s|^{\frac{2\sigma}{\theta}}
	\| u\|_{L^{\infty}([s,t];H^1)}^{2\sigma} \|w_k\|_{L^{p}([s,t];L^{q})} \\
	\label{eq:generalk3}
	& & \hspace{1cm} + \;C |t-s| \|w_k \|_{L^{\infty}([s,t];L^{2})}\\
        \label{eq:generalk4}
	& & \hspace{1cm} + \;C \!\!\!\!\!\!\!\!\!\!\! \sum_{\substack{0< j < k \\ 
                        1\le l_1, \dots,l_{j+1} \le k-1 \\
                        l_1+\dots +l_{j+1}=k}
                        }        
        \!\!\!\!\!\!\!\!\!\!\!
        |t-s|^{\frac{2\sigma-j}{\theta}}
	\| u\|_{L^{\infty}([s,t];H^1)}^{2\sigma-j} \|w_{l_1}\|_{L^{p}([s,t];L^{q})}
	\dots \|w_{l_{j+1}}\|_{L^{p}([s,t];L^{q})}\\
        \label{eq:generalk5}
	& & \hspace{1cm} + \;C \!\!\! \sum_{0\le j \le k-1}\!\!\! |t-s| 
        \|w_j \|_{L^{\infty}([s,t];L^{2})}.
\end{eqnarray}
This completely generalizes the $k=2$ case and no qualitatively new terms appear anymore.

To conclude the proof, we only need to argue that for 
\eqref{eq:generalk2} and \eqref{eq:generalk3} we can do
\begin{multline*}
	C |t-s|^{\frac{2\sigma}{\theta}}
	\| u\|_{L^{\infty}([s,t];H^1)}^{2\sigma} \|w_k\|_{L^{p}([s,t];L^{q})} 
	+ \;C |t-s| \|w_k \|_{L^{\infty}([s,t];L^{2})}
        \le \\
	 C \tau^{\alpha}  {\rm e}^{Ct} \|w_k\|_{L^p([s,t];L^q)\cap L^{\infty}([s,t];L^2)}, 
\end{multline*}
by making $|t-s|=\tau \le \tau_0$, $\alpha=\frac{2\sigma}{\theta}$ and
using the exponential growth of the $\Sigma$ norm, from Theorem~\ref{theo:exp}, 
to bound the $\| u\|_{L^{\infty}([s,t];H^1)}$ norm (the $H^1$ subcritical condition $\sigma < \frac{2}{d-2}$
guarantees that $0<\alpha <1$). Whereas, for \eqref{eq:generalk4} and
\eqref{eq:generalk5}, these involve exclusively the norms of the previous induction
steps $\le k-1$, whose growth is known by the induction hypothesis, and which can thus
be bounded by a non-decreasing double exponential function $f(t)=C{\rm e}^{{\rm e}^{Ct}}$.

Therefore, we have finally established that this general Strichartz type inequality for the norm 
of the derivatives and momenta of order $k$, $\|w_k\|_{L^p([s,t];L^q)
  \cap {L^{\infty}([s,t];L^2)}}$, suits exactly 
the hypotheses of Lemma~\ref{lemma:main} from which we infer its double exponential
growth, ending the proof.

\begin{remark}
\label{remark:conf}
 This result is a corollary of the exponential growth rate of the $\Sigma$ norm, coming from
 Theorem~\ref{theo:exp}, because that rate controls the norm $\| u\|_{L^{\infty}([s,t];H^1)}$ in \eqref{eq:generalk2}
 and consequently 
 the estimate \eqref{eq:induction} in the hypotheses of Lemma~\ref{lemma:main}. If the 
 norm $\| u\|_{L^{\infty}([s,t];H^1)}$, however, is known to grow
 at a different rate, then a different final result is obtained for the growth rate of the norms of
 the higher order derivatives and momenta. In particular, for confining time independent harmonic 
 potentials of the type \eqref{eq:confining}, the conservation \eqref{eq:evolenergy} of the energy
 \eqref{eq:energy} implies that $u \in L^{\infty}(\R, \Sigma)$ and thus that this crucial term  is uniformly bounded in time
 \begin{equation*}
    \| u\|_{L^{\infty}([s,t];H^1)} \le C, \quad \forall_{0\le s\le t}.
 \end{equation*}
 Then, if one were to follow the exact same steps of the previous corollary's proof, the only 
 difference would occur in that, instead of an exponential term in \eqref{eq:induction} in 
 the hypotheses of Lemma~\ref{lemma:main}, we would now have 
 \begin{equation*}
  \|w\|_{L^p(I;L^q)\cap L^{\infty}(I;L^2)}\le C \|w(s)\|_{L^2} + C\tau^\alpha 
    \|w\|_{L^p(I;L^q)\cap L^{\infty}(I;L^2)}+f(t),
 \end{equation*}
 that yields a single exponential growth
 \begin{equation*}
  \|w\|_{L^p([0,t];L^q)\cap L^{\infty}([0,t];L^2)}\le C_1{{\rm e}^{C_1t}}
  (\|w(0)\|_{L^2}
  +f(t)),\quad \forall
  t\ge 0.
 \end{equation*}
 Indeed, in the proof of Lemma~\ref{lemma:main}, the size of the small time intervals $\tau\le
 \tau_0$ would now be just a uniform constant and not depend on $t$. So that, the number of these
 intervals $N \sim t/\tau$ would merely be proportional to $t$, and not exponential. As everything
 else follows exactly the same way as in the general case, proved above, this then implies the final 
 result, of the single exponential growth of norms of the higher derivatives and momenta, in the
 time independent confining potential case. 
\end{remark}

\section{Asymptotically vanishing potential: proof of
  Theorem~\ref{theo:decay}} 
\label{sec:decay}

The proof is based on a lens transform as in \cite{Ca11}. The main
idea is that the decay of $\Omega$ as $t\to \infty$ allows us to avoid
the compactification of time, which is one of the features of the lens
transform in the case $\Omega=1$ (see e.g. \cite{CaM3AS,TaoLens}). 

\subsection{Lens transform}
\label{sec:lens}

Since some adaptations will be needed, we resume the approach
presented in \cite[Section~4]{Ca11}. Suppose that $v$ solves a 
non-autonomous equation
\begin{equation}
  \label{eq:nls-na}
  i\d_t v+\frac{1}{2}\Delta v = H(t)\lvert v\rvert^{2\si}v.
\end{equation}
We want $v$ and $u$ (solution of \eqref{eq:nlsp}) to be related by the
formula
\begin{equation}
  \label{eq:lensgen}
  u(t,x)=\frac{1}{b(t)^{d/2}}v\(\zeta(t),\frac{x}{b(t)}\)
{\rm e}^{\frac{i}{2}a(t)|x|^2}, 
\end{equation}
with $a,b,\zeta$ real-valued, and for some time $t_0\ge 0$,
\begin{equation}\label{eq:cilens}
  b(t_0)>0;\quad 
\zeta(t_0)>0. 
\end{equation}

Apply the Schr\"odinger differential operator to the formula
\eqref{eq:lensgen}, and identify the terms so that $u$ solves
\eqref{eq:nlsp}. We find: 
\begin{equation}\label{eq:systlens}
  \dot b = ab\quad ;\quad \dot a+a^2 +\Omega=0\quad ;\quad
\dot \zeta = \frac{1}{b^2}\quad 
;\quad b(t)^{d\si-2}H\(\zeta(t)\)=1. 
\end{equation}
Introduce a solution to 
\begin{equation}\label{eq:pair}
\ddot \mu + \Omega(t)\mu =0 ;\quad 
    \ddot \nu + \Omega(t)\nu =0,
\end{equation}
such that the (constant) Wronskian is $W:= \nu \dot
\mu - \dot\nu \mu\equiv 1$. The solutions to \eqref{eq:systlens} can
be expressed as 
\begin{equation*}
  a=\frac{\dot \nu}{\nu}\quad ;\quad b= \nu\quad ;\quad
  \zeta=\frac{\mu}{\nu}. 
\end{equation*}
Note that $\zeta$ is locally invertible, since $\zeta(t_0)>0$ and 
\begin{equation*}
  \dot \zeta = \frac{1}{b^2}=\frac{1}{\nu^2},\text{ hence }\dot \zeta(t_0)>0.
\end{equation*}
Therefore, the lens transform is locally
invertible. Moreover, we can write, as long as $\nu>0$,
\begin{equation*}
  H(t) = b\(\zeta^{-1}(t)\)^{2-d\si}=
  \nu\(\(\frac{\mu}{\nu}\)^{-1}(t)\)^{2-d\si}. 
\end{equation*}

\subsection{A particular fundamental solution}
\label{sec:fundamental}

The idea is then to construct a suitable solution $(\mu,\nu)$, so that
$\zeta$ is invertible in a neighborhood of $t=+\infty$. This
scattering problem is solved by adapting \cite[Lemma~A.1.2]{DG}:
\begin{lemma}
  Suppose that there exists $\gamma>2$ such that $|\Omega(t)|\lesssim
  \<t\>^{-\gamma}$. Then there exist $\mu_\infty,\nu_\infty$ solving
  \begin{equation*}
  \ddot \mu_\infty + \Omega(t)\mu_\infty =0 ;\quad   \ddot \nu_\infty +
  \Omega(t)\nu_\infty =0, 
  \end{equation*}
with $\nu_\infty(T)=1$, $\mu_\infty(T)=T$ for some $T>0$, and such
that, as $t\to \infty$,  
\begin{align}
\label{eq:nuinf}&  |\nu_\infty(t)-1|=\O\(\frac{1}{t^{\gamma-2}}\),\quad |\dot
  \nu_\infty(t)|= \O\(\frac{1}{t^{\gamma-1}}\),\\
\label{eq:muinf}&  |\mu_\infty(t)-t|=\O\(t^{3-\gamma}\),\quad |\dot
  \mu_\infty(t)-1|= \O\(\frac{1}{t^{\gamma-2}}\).
\end{align}
The Wronskian of $\nu_\infty$ and $\mu_\infty$ is $W:= \nu_\infty \dot
\mu_\infty - \dot\nu_\infty \mu_\infty \equiv 1$. 
\end{lemma}
\begin{proof}
  For $T>0$, consider the problem
  \begin{equation*}
    \ddot z + \Omega(t)z +g(t)=0;\quad z(T)=0;\quad \lim_{t \to +\infty}
    \dot z(t)=0.
  \end{equation*}
The integral formulation of this problem reads
\begin{equation}\label{eq:int}
  z = r_T + {\mathcal P}_T z,
\end{equation}
with 
\begin{align*}
  r_T(t)& = \int_T^t (s-T)g(s)\dd s +(t-T)\int_t^\infty g(s)\dd s,\\
{\mathcal P}_T z(t)&= \int_T^t (s-T)\Omega(s)z(s)\dd s +
(t-T)\int_t^\infty \Omega(s)z(s)\dd s. 
\end{align*}
Consider the two Banach spaces
\begin{align*}
  Z^0_T & = \{z\in C([T,\infty);\R)\ |\ \|z\|_0:=\sup_{t\ge
    T}|z(t)|<\infty\}, \\
Z^1_T & = \left\{z\in C([T,\infty);\R)\ |\ \|z\|_1:=\sup_{t\ge
    T}\frac{|z(t)|}{t-T}<\infty\right\}.
\end{align*}
We readily check that ${\mathcal P}_T$ is bounded on $Z_T^0$ and
$Z_T^1$, respectively, and that its norm on either of these spaces
equals
\begin{equation*}
  \int_T^\infty (t-T)\Omega(t)\dd t. 
\end{equation*}
By assumption, it is estimated by
  \begin{equation*}
  \int_T^\infty (t-T)\Omega(t)\dd t\lesssim \int_T^\infty\frac{\dd
    t}{\<t\>^{\gamma-1}} + T\int_T^\infty\frac{\dd
    t}{\<t\>^{\gamma}}\lesssim T^{2-\gamma}\Tend T {+\infty} 0.
\end{equation*}
Fixing $T$ sufficiently large, this norm is smaller than one, and
\eqref{eq:int} has a unique solution, given by
\begin{equation*}
  z= \(1-{\mathcal P}_T\)^{-1}r_T \in Z^{j}_T,
\end{equation*}
with $j=0$ or $1$, according to the case considered, provided that
$r_T\in Z^{j}_T$. 
\smallbreak

In view of the statement of the proposition, we start by constructing
$\nu_\infty$: $z=\nu_\infty-1$ must solve
\begin{equation*}
  \ddot z + \Omega(t)z +\Omega(t)=0.
\end{equation*}
Therefore, we use the above general result with $g=\Omega$: the
function $r_T$ belongs to $Z^0_T$, hence a function $\nu_\infty\in
Z^0_T$ (since $1\in Z^0_T$). Since $r_T(t) = \O(t^{2-\gamma})$ as
$t\to \infty$, we readily get
\begin{equation*}
  |z(t)|=|\nu_\infty(t)-1|=\sum_{j=0}^\infty {\mathcal P}_T^j r_T(t)=
  \O\(\frac{1}{t^{\gamma-2}}\).  
\end{equation*}
Differentiating \eqref{eq:int}, we infer
\begin{equation*}
  \dot z(t) = \int_t^\infty g(s)\dd s + \int_t^\infty \Omega(s)z(s)\dd
  s =\int_t^\infty \Omega(s)\dd s + \int_t^\infty \Omega(s)z(s)\dd
  s  .
\end{equation*}
Since $z\in Z^0_T$, we obtain, for $t\ge T$, 
\begin{equation*}
  |\dot z(t)| = |\dot \nu_\infty(t)| = \O\(\frac{1}{t^{\gamma-1}}\). 
\end{equation*}
In the case of $\mu_\infty$, we work in $Z^1_T$ instead:
$z(t)=\mu_\infty(t)-t$ must satisfy
\begin{equation*}
  \ddot z + \Omega(t)z +t\Omega(t)=0,
\end{equation*}
that is, $g(t)=t\Omega(t)$. Since we now have the pointwise estimate
\begin{equation*}
  r_T(t) = \O\(T^{3-\gamma}\)+\O\(t^{3-\gamma}\), 
\end{equation*}
we have $r_T\in Z^1_T$, and for $T$ sufficiently large,
\begin{equation*}
  z= \(1-{\mathcal P}_T\)^{-1}r_T \in Z^{1}_T.
\end{equation*}
Proceeding like for $\nu_\infty$, we infer
\begin{equation*}
  |\mu_\infty(t)-t|=\O\(t^{3-\gamma}\),\quad |\dot
  \mu_\infty(t)-1|= \O\(\frac{1}{t^{\gamma-2}}\).
\end{equation*}
Finally, the Wronskian $W$ does not depend on time, and goes to one as
$t$ goes to infinity, so $W\equiv 1$. 
\end{proof}

\subsection{End of the proof}
\label{sec:end}

In the pseudo-conformal invariant case $\si=2/d$, we see that
\eqref{eq:nls-na} is the standard autonomous equation, for which
there is scattering. In addition, the Sobolev norms are bounded in
time, and the momenta grow polynomially in time as in the statement of
Theorem~\ref{theo:decay}, as established in \cite{Wa04} (see also the
appendix in \cite{Ca11}).  In view of the asymptotic properties of
$\mu_\infty$ and $\nu_\infty$, $u$ satisfies the same properties. We
will be more precise concerning these statement by considering more
generally the case $\si\ge 2/d$, where \eqref{eq:nls-na} must be
thought of as a non-autonomous equation. We readily check the analogue
of the conservation of energy and the pseudo-conformal conservation
law in the present case.
  Let $J(t)=x+it\nabla$. The solution to \eqref{eq:nls-na} satisfies
  \begin{align}
\label{eq:energy-na}\frac{\dd }{\dd t}\(\frac{1}{2}\|\nabla v\|_{L^2}^2
    +\frac{H(t)}{\si+1}\|v(t)\|_{L^{2\si+2}}^{2\si+2}\) &= \frac{\dot
      H(t)}{\si+1}\|v(t)\|_{L^{2\si+2}}^{2\si+2}\\ 
 \label{eq:pseudo-na}   \frac{\dd }{\dd t}\(\frac{1}{2}\|J(t)v\|_{L^2}^2
    +\frac{t^2H(t)}{\si+1}\|v(t)\|_{L^{2\si+2}}^{2\si+2}\) &=
    \frac{tH(t)}{\si+1}(2-d\si)\|v(t)\|_{L^{2\si+2}}^{2\si+2}\\
\notag &\quad +\frac{t^2\dot H(t)}{\si+1}\|v(t)\|_{L^{2\si+2}}^{2\si+2},
  \end{align}
where we recall that
\begin{equation*}
  H(t)
  =\nu_\infty\(\(\frac{\mu_\infty}{\nu_\infty}\)^{-1}(t)\)^{2-d\si}
\end{equation*}
is well-defined for $t\ge t_0$ sufficiently large:
\begin{equation*}
  |\nu_\infty(t)-1|\le \frac{1}{2}\quad \text{for } t\ge t_0. 
\end{equation*}
In addition,
\begin{equation*}
 \dot H(t) = (2-d\si)
 \nu_\infty\(\(\frac{\mu_\infty}{\nu_\infty}\)^{-1}(t)\)^{3 -d\si}\dot
 \nu_\infty\(\(\frac{\mu_\infty}{\nu_\infty}\)^{-1}(t)\), 
\end{equation*}
hence $\dot H(t) = \O\(\frac{1}{t^{\gamma-1}}\)$ as $t\to \infty$. 
We infer from \cite[Lemma~3.1]{CW92} (see also
\cite[Theorem~4.11.1]{CazCourant}) that 
\eqref{eq:nls-na} has a unique, global solution $v\in
C(\R_+;\Sigma)$. Therefore, the lens transform \eqref{eq:lensgen} is
well-defined and bijective in a neighborhood of $t=+\infty$.
\smallbreak

Set, for $t$ sufficiently large,
\begin{equation*}
  y(t) = \frac{t^2H(t)}{\si+1}\|v(t)\|_{L^{2\si+2}}^{2\si+2}.
\end{equation*}
The relation \eqref{eq:pseudo-na} yields
\begin{equation*}
  y(t) \le C\(\|v(t_0)\|_\Sigma\) +\int_{t_0}^t \frac{|\dot
    H(s)|}{H(s)}y(s)\dd s\lesssim 1 + \int_{t_0}^t y(s)\frac{\dd
    s}{s^{\gamma-1}} , 
\end{equation*}
and we infer from Gronwall lemma that $y \in
L^\infty([t_0,\infty))$. Using \eqref{eq:pseudo-na} again, we deduce
that $J(t)v \in L^\infty([t_0,\infty);L^2)$. Since 
\begin{equation*}
  J(t) = it {\rm e}^{i\frac{|x|^2}{2t}}\nabla\( {\rm e}^{-i\frac{|x|^2}{2t}} \cdot\),
\end{equation*}
Gagliardo--Nirenberg inequality yields, for $2\le r\le 2/(d-2)$ ($r\le
\infty$ if $d=1$, $r<\infty$ if $d=2$), 
\begin{equation}\label{eq:Jdecay}
  \|v(t)\|_{L^r}\lesssim
  \frac{1}{|t|^{\delta}}\|v\|_{L^2}^{1-\delta}
  \|J(t)v\|_{L^2}^\delta ,\quad \text{where
  }\delta=d\(\frac{1}{2}-\frac{1}{r}\), 
\end{equation}
hence a decay rate (in time) for Lebesgue norms (in space). Mimicking
the proof of \cite[Proposition~A.4]{Ca11}, we conclude:
\begin{proposition}\label{prop:growth}
  Let $\si \ge 2/d$ be an integer, with $\si\le 2/(d-2)$ if $d\ge
  3$. Suppose $v_{\mid t=0}\in \Sigma^k$ for some $k\in \N$, $k\ge
  1$. Then $v\in C(\R;\Sigma^k)$. In addition, 
for all admissible pair $(p,q)$, $v\in
L^{p}(\R;W^{k,q}(\R^d))$, and 
  \begin{equation*}
    \forall \alpha\in \N^d, \ |\alpha|\le k,\quad \left\lVert x^\alpha
      v\right\rVert_{L^{p}([0,t];L^q)} \lesssim \<t\>^{|\alpha|}. 
  \end{equation*}
\end{proposition}
Since
\begin{equation}\label{eq:lensfinal}
  u(t,x)=\frac{1}{\nu_\infty(t)^{d/2}}v\(\frac{\mu_\infty(t)}{\nu_\infty(t)}
  ,\frac{x}{\nu_\infty(t)}\)
{\rm e}^{\frac{i}{2}|x|^2\dot\nu_\infty(t)/\nu_\infty(t)}, 
\end{equation}
we have 
\begin{align*}
  \|u(t)\|_{H^k}&\lesssim \frac{1}{|\nu_\infty(t)|}\left\|
    v\(\frac{\mu_\infty(t)}{\nu_\infty(t)}\) \right\|_{H^k} + |\dot
  \nu_\infty(t)|^k \left\lVert \lvert
    x\rvert^kv\(\frac{\mu_\infty(t)}{\nu_\infty(t)}\)
    \right\rVert_{L^2},  \\
\left\lVert \lvert
    x\rvert^k u(t)\right\rVert_{L^2}&\lesssim |\nu_\infty(t)|^k
  \left\lVert \lvert  
    x\rvert^kv\(\frac{\mu_\infty(t)}{\nu_\infty(t)}\)
    \right\rVert_{L^2}.
\end{align*}
Gathering \eqref{eq:nuinf}, \eqref{eq:muinf} and
Proposition~\ref{prop:growth} together,  we obtain
Theorem~\ref{theo:decay}, up to the scattering result. 
\smallbreak

In view of \eqref{eq:Jdecay}, scattering for \eqref{eq:nls-na} follows
from the standard approach: from Duhamel's formula,
\begin{equation*}
  {\rm e}^{-i\frac{t}{2}\Delta}v(t)- {\rm
    e}^{-i\frac{\tau}{2}\Delta}v(\tau) = -i\int_\tau^t {\rm
    e}^{-i\frac{s}{2}\Delta} \(H(s)|v(s)|^{2\si}v(s)\)ds,
\end{equation*}
we infer that $\({\rm e}^{-i\frac{t}{2}\Delta}v(t)\)_{t\ge 0}$ is a
Cauchy sequence in $\Sigma$ as $t\to \infty$, hence converges to some $v_+\in
\Sigma$. Taking  into account \eqref{eq:nuinf}, \eqref{eq:muinf},
\eqref{eq:lensfinal} and the asymptotics for ${\rm
  e}^{i\frac{t}{2}\Delta}$ recalled in \eqref{eq:scatt}, the last
point of Theorem~\ref{theo:decay} follows. 

\begin{remark}[Sharpness of the decay assumption on $\Omega$]\label{rem:sharp}
  The key property that we have used to define a lens transform in the
  neighborhood of $t=+\infty$ is that the function $\zeta=\mu/\nu$ is
  bijective from $[T,\infty)$ to $[\zeta(T),\infty)$ for $T$
  sufficiently large, where $(\mu,\nu)$ solves \eqref{eq:pair} with a
  Wronskian equal to one. Note that if we had, instead of the above
  property, $\zeta(T)<0$ and $\zeta$ bijective from $[T,\infty)$ to
  $(-\infty,\zeta(T)]$, it would be straightforward to adapt the above
  analysis (simply replace $\nu$ with $-\nu$). The point is that
  unlike, in the case $\Omega=1$, the lens transform does not
  compactify time, because $\nu$ only has a finite number of zeroes,
  and $\dot \zeta=1/\nu^2$ ($\nu(t) = \cos t$ in the case
  $\Omega=1$). This property would remain under the mere 
  assumption (see for instance \cite[Chapter~XI, Theorem~7.1]{Hartman})
  \begin{equation*}
    -\infty\le \limsup_{t\to +\infty} t^2\Omega(t) < \frac{1}{4}.
  \end{equation*}
In the opposite case, $\nu$ has infinitely many oscillations, and even
if it may sound sensible to expect most of Theorem~\ref{theo:decay} to
remain valid, the last conclusion (scattering) should fail: the
dynamics has a different nature.  
\end{remark}

\noindent {\bf Acknowledgments}. J. Drumond Silva would like to thank the kind
hospitality of the Department of Mathematics at the Universit\'e Montpellier 2, where part of this work was developed.

 \end{document}